\documentclass[11pt,a4paper]{amsart}

\usepackage{amssymb,amsmath,amsxtra}
\usepackage{xcolor}
\usepackage{hyperref}
\hypersetup{
    colorlinks,
    linkcolor={red!50!black},
    citecolor={blue!50!black},
    urlcolor={blue!80!black}
}
\usepackage[margin=2.5cm]{geometry}
\usepackage{cite}
\usepackage{graphicx}
\usepackage[curve]{xypic}

\newtheorem{theorem}{Theorem}[section]
\newtheorem{lemma}[theorem]{Lemma}
\newtheorem{corollary}[theorem]{Corollary}
\newtheorem{proposition}[theorem]{Proposition}
\theoremstyle{definition}
\newtheorem{definition}[theorem]{Definition}
\newtheorem{example}[theorem]{Example}
\newtheorem{notation}[theorem]{Notation}

\newtheorem{remark}[theorem]{Remark}

\numberwithin{equation}{section}

\newcommand{\cA}{\mathcal{A}}

\DeclareMathOperator{\Aut}{Aut}
\DeclareMathOperator{\ch}{ch}

\DeclareMathOperator{\Bl}{Bl}
\newcommand{\Cstar}{\CC^\times}

\newcommand{\CC}{\mathbb{C}}
\DeclareMathOperator{\Cone}{Cone}
\newcommand{\Dstack}[2]{D^b_{#2}\!\left(#1\right)}

\newcommand{\Fix}{\text{\rm Fix}}

\DeclareMathOperator{\Hom}{Hom}
\newcommand{\tti}{\mathtt{i}}

\newcommand{\LL}{\mathbb{L}}
\DeclareMathOperator{\Lie}{Lie}

\newcommand{\cO}{\mathcal{O}}
\newcommand{\PP}{\mathbb{P}}
\newcommand{\QQ}{\mathbb{Q}}
\newcommand{\RR}{\mathbb{R}}
\newcommand{\FM}{\mathbb{FM}}

\newcommand{\ZZ}{\mathbb{Z}}

\newcommand{\hS}{\widehat{S}}

\newcommand{\tD}{\widetilde{D}} 
\newcommand{\tomega}{{\tilde{\omega}}} 
\newcommand{\tX}{\widetilde{X}}

\newcommand{\Frac}{\operatorname{Frac}}

\newcommand{\Spec}{\operatorname{Spec}}

\newcommand{\tch}{\operatorname{\widetilde{\ch}}}
\newcommand{\rank}{\operatorname{rank}} 
\newcommand{\Td}{\operatorname{Td}} 
\newcommand{\tTd}{\widetilde{\Td}}

\newcommand{\Tr}{\operatorname{Tr}}

\newcommand{\cF}{\mathcal{F}} 
\newcommand{\cE}{\mathcal{E}}

\newcommand{\sfB}{\mathsf{B}}

\begin{document}

\title[$K$-Theoretic and Categorical Properties of Toric Deligne--Mumford Stacks]
{$K$-Theoretic and Categorical Properties of\\ Toric Deligne--Mumford Stacks}

\author[Coates]{Tom Coates}
\address{Department of Mathematics\\
Imperial College London\\
180 Queen's Gate\\
London SW7 2AZ 
\\UK}
\email{t.coates@imperial.ac.uk}
\email{edward.segal04@imperial.ac.uk}

\author[Iritani]{Hiroshi Iritani}
\address{Department of Mathematics\\
Graduate School of Science\\
Kyoto University\\
Oiwake-cho\\
Kitashirakawa\\
Sakyo-ku\\
Kyoto, 606-8502\\
Japan}
\email{iritani@math.kyoto-u.ac.jp}

\author[Jiang]{Yunfeng Jiang}
\address{Department of Mathematics\\
University of Kansas\\
1460 Jayhawk Boulevard\\
Lawrence, Kansas 66045-7594\\
USA}
\email{y.jiang@ku.edu}

\author[Segal]{Ed Segal}

\thanks{First published in Pure and Applied Mathematics Quarterly in 
Volume 11, Number 2, 239--266 (2015), 
published by International Press.}



\date{}

\begin{abstract}
We prove the following results for toric Deligne--Mumford stacks, under minimal compactness hypotheses: the Localization Theorem in equivariant $K$-theory; the equivariant Hirzebruch--Riemann--Roch theorem; the Fourier--Mukai transformation associated to a crepant toric wall-crossing gives an equivariant derived equivalence.
\end{abstract}

\maketitle

\section{Introduction}

We establish various basic geometric properties of toric Deligne--Mumford stacks under minimal compactness hypotheses.  
This is a companion paper to~\cite{Coates--Iritani--Jiang}: the results here are used there in the proof of 
the Crepant Transformation Conjecture for  toric Deligne--Mumford stacks, and we expect that they will also be useful elsewhere.   We establish the equivariant Hirzebruch--Riemann--Roch theorem and the Localization Theorem in equivariant $K$-theory, two of the fundamental tools in equivariant topology, for toric Deligne--Mumford stacks without requiring compactness.  We also give an equivariant generalization of a celebrated result of Kawamata, that $K$-equivalent toric Deligne--Mumford stacks are derived equivalent, and exhibit an explicit Fourier--Mukai kernel that implements this equivalence.  This latter result plays an essential role in the proof of the Crepant Transformation Conjecture~\cite{Coates--Iritani--Jiang}: it implies that the transformation which controls the change in quantum cohomology under a crepant transformation (between toric Deligne--Mumford stacks or complete intersections therein) is, in an appropriate sense, a linear symplectic isomorphism. None of the results proved here are surprising, but we were unable to find proofs of them, at this level of generality, in the  literature.  
Note in particular that our formulation (equation~\ref{eq:equiv_HRR} below) of the equivariant Hirzebruch--Riemann--Roch theorem makes sense for arbitrary (not just toric) smooth Deligne--Mumford stacks with torus action under mild hypotheses; we believe this formulation to be new. 


We consider toric Deligne--Mumford stacks $X$ such that:
\begin{enumerate}
\item the torus-fixed set $X^T$ is non-empty; and
\item the coarse moduli space $|X|$ is semi-projective, i.e.~$|X|$ is projective over the affinization $\Spec(H^0(|X|,\cO))$.
\end{enumerate}
These conditions are equivalent to demanding that $X$ arise as the GIT quotient $\big[\CC^m /\!\!/_\omega K\big]$ of a vector space by the linear action of a complex torus $K$, as in \S\ref{sec:GIT} below.  The action of $T = (\Cstar)^m$ on $\CC^m$ descends to give an ineffective action of $T$ on $X$, as well as an effective action of the quotient torus $T/K$ on $X$. 

The Localization Theorem in equivariant $K$-theory and the equivariant index theorem were first proved for the topological $K$-theory of $G$-spaces and $G$-manifolds by Atiyah and Segal~\cite{Segal,Atiyah--Segal}.  Similar results were established in algebraic $K$-theory by Nielsen~\cite{Nielsen} and Thomason~\cite{Thomason:1,Thomason:2,Thomason:3}.  Index theorems have been proven for compact orbifolds by Kawasaki~\cite{Kawasaki} and for proper Deligne--Mumford stacks by Toen~\cite{Toen}; an equivariant index theorem for compact orbifolds was proved by Vergne~\cite{Vergne}.  In \S\S\ref{sec:HRR}--\ref{sec:localization} we prove an equivariant index theorem for toric Deligne--Mumford stacks, without requiring properness, using methods and results of Atiyah--Segal and Thomason.  

In \S\ref{sec:derived_equivalence} we prove that the Fourier--Mukai functor associated to the $K$-equivalence
\[
\xymatrix{
  & \widetilde{X} \ar[ld]_{f_+} \ar[rd]^{f_-} \\
  X_+ \ar@{-->}[rr]^{\varphi} && X_- 
}
\]
determined by a crepant wall-crossing of toric GIT quotients gives an equivalence between the equivariant derived categories of $X_\pm$.
This is an equivariant generalization of a result of Kawamata \cite{Kawamata}, with a different proof: we use the theory developed by Halpern-Leistner~\cite{Halpern-Leistner} and Ballard--Favero--Katzarkov~\cite{Ballard--Favero--Katzarkov} which relates derived categories to variation of GIT.

Toric Deligne--Mumford stacks were introduced by Borisov--Chen--Smith~\cite{Borisov--Chen--Smith}, who described them in terms of \emph{stacky fans}.  They have also been studied by Jiang~\cite{Jiang}, who introduced the notion of an \emph{extended stacky fan}.  Our approach here, where we treat toric Deligne--Mumford stacks as GIT quotients $\big[\CC^m /\!\!/_\omega K\big]$, is equivalent to the approach via (extended) stacky fans.  This is explained in~\cite[\S4.2]{Coates--Iritani--Jiang}.

\section{The Hirzebruch--Riemann--Roch Formula} 
\label{sec:HRR}

Let $X$ be a smooth Deligne-Mumford stack with a torus $T$ action, which satisfies  the following properties:
\begin{itemize} 
\item[(P1)] the coarse moduli space $|X|$ is semi-projective; \label{property:1}
\item[(P2)] all the $T$-weights appearing in the $T$-representation $H^0(X,\cO)$ are contained in a strictly convex cone in $\Lie(T)^*$, and the $T$-invariant subspace $H^0(X,\cO)^T$ is $\CC$. \label{property:2}
\end{itemize}
\noindent These properties together imply, for example, that the fixed set $X^T$ is compact.  
As we will see, these properties allow us to define the equivariant index of coherent sheaves on $X$, and to state the equivariant Hirzebruch--Riemann--Roch formula (equation~\ref{eq:equiv_HRR} below).  In~\S\ref{sec:localization} below we prove this Hirzebruch--Riemann--Roch formula for toric Deligne--Mumford stacks, using the Localization Theorem in equivariant $K$-theory.

Let $|X|$ denote the coarse moduli space of $X$ and let $IX=X\times_{|X|}X$ denote the inertia stack of $X$. 
The stack $IX$ consists of pairs $(x,g)$ with $x\in X$ 
and $g\in \Aut_X(x)$.  We write $H^{\bullet \bullet}_{T}(IX) := \prod_{p} H^{2p}_{T}(IX)$, where the superscript `${}^{\bullet \bullet}$' is to indicate that we take the direct product of cohomology groups rather than the direct sum; it does not indicate a double grading.  Let $K_T^0(X)$ denote the Grothendieck group of $T$-equivariant vector bundles on $X$.  We introduce an orbifold Chern character map
\[
\tch \colon K^0_T(X) \to H_{T}^{\bullet\bullet}(IX)
\]
as follows.  Let $IX = \bigsqcup_{v \in \sfB} X_v$ be the decomposition of the inertia stack $IX$ into connected components, where the set $\sfB$ indexes the connected components.  (If $X$ is a toric Deligne--Mumford stack then we can take $\sfB$ to be the box of the stacky fan that defines $X$.)
Let $q_v \colon X_v \to X$ be the natural map, and let $E$ be a $T$-equivariant vector bundle on $X$.  The stabilizer $g_v$ along $X_v$ acts on the vector bundle $q_v^* E \to X_v$, giving an eigenbundle decomposition
\[
q_v^* E = \bigoplus_{\theta \in \Theta(v)} E_{v,\theta}
\]
where $\Theta(v)$ denotes the set of rational numbers $\theta \in [0,1)$ such that $\exp(2\pi\tti \theta)$ is an eigenvalue of the $g_v$-action on $q_v^*(E)$, and $E_{v,\theta}$ is the eigenbundle corresponding to $\theta$.  The equivariant Chern character is defined to be
\[
\tch(E) = \bigoplus_{v\in \sfB} \sum_{\theta \in \Theta(v)} e^{2\pi\tti \theta} 
\ch^T(E_{v,\theta}) 
\]
where $\ch^T(E_{v,\theta})\in H^{\bullet\bullet}_T(X_v)$ is the usual $T$-equivariant Chern character of the $T$-equivariant vector bundle $E_{v,\theta} \to X_v$.   
Let $\delta_{v,\theta,i}$, $1 \leq i \leq \rank(E_{v,\theta})$ be the $T$-equivariant Chern roots of $E_{v,\theta}$, so that $c^T(E_{v,\theta}) = \prod_i (1+ \delta_{v,\theta,i})$. 
These Chern roots are not actual cohomology classes, but  symmetric polynomials in the Chern roots make sense as equivariant cohomology classes on $X_v$.  The $T$-equivariant orbifold Todd class $\tTd(E) \in H^{\bullet\bullet}_{T}(IX)$ is defined to be: 
\[
\tTd(E) = \bigoplus_{v\in \sfB} 
\left( \prod_{\substack{\theta \in \Theta(v)\\ \theta \ne 0}} \prod_{i=1}^{\rank(E_{v,\theta})} 
\frac{1}{1- e^{-2\pi\tti \theta} e^{-\delta_{v,\theta,i}}}\right) 
\prod_{i=1}^{\rank E_{v,0}} 
\frac{\delta_{v,0,i}}{1-e^{-\delta_{v,0,i}}}
\]
We write $\tTd_X = \tTd(TX)$ for the orbifold Todd class of the tangent bundle.

Property~(P2) gives that all of the $T$-weights of $H^0(X,\cO)$ lie in a strictly convex cone in $ \Lie(T)^*$.  After changing the identification of $T$ with $(\Cstar)^m$ if necessary, we may assume that this cone is contained within the cone spanned by the standard characters $\lambda_1,\ldots,\lambda_m$ in $H^2_T({\rm pt}) = \Lie(T)^*$, where $\lambda_j \colon T \to \Cstar$ is given by projection to the $j$th factor of $T = (\Cstar)^m$.  The Chern character 
\[
\ch^T \colon K_T^0({\rm pt}) \to 
\ZZ[e^{\pm\lambda}] := 
\ZZ[e^{\pm \lambda_1},\dots,e^{\pm \lambda_m}]
\subset H^{\bullet\bullet}_T({\rm pt})
\]
sends the irreducible representation of weight $\lambda_i$ 
to $e^{\lambda_i}$. 
The $T$-representation $H^0(X,\cO)$ is infinite dimensional, 
but each weight piece is finite dimensional. 
Thus we have a well-defined character $\ch^T(H^0(X,\cO))$ in 
$\ZZ[\![e^{\lambda}]\!]:=\ZZ[\![e^{\lambda_1},\dots,e^{\lambda_m}]\!]$. 
More generally, if $V$ is a locally finite $T$-representation that is finitely 
generated as an $H^0(X,\cO)$-module, the character 
$\ch^T(V)$ lies in $\ZZ[\![e^{\lambda}]\!][e^{-\lambda}] := 
\ZZ[\![e^{\lambda}]\!]
[e^{-\lambda_1},\dots,e^{-\lambda_m}]$. 
An important fact is that \emph{$\ch^T(V)$ becomes a rational function 
in $e^{\lambda_1},\dots,e^{\lambda_m}$} for such $V$, see \cite{Thomason:3}. 
In other words, $\ch^T(V)$ lies in: 
\[
\ZZ[\![e^{\lambda}]\!][e^{-\lambda}]_{\rm rat}
:= \left\{ f \in \ZZ[\![e^{\lambda}]\!]
[e^{-\lambda}] : 
\begin{array}{l}
\text{$f$ is the Laurent expansion of a rational function} \\ 
\text{in $\CC(e^{\lambda_1},\dots,
e^{\lambda_m})$ at $e^{\lambda_1} = 
\cdots = e^{\lambda_m}=0$}
\end{array}\right\}
\]
For a $T$-equivariant vector bundle $E$ on $X$, 
the cohomology groups $H^i(X,E)$ are finitely generated 
$H^0(X,\cO)$-modules since $|X|$ is semi-projective. 
Therefore we can define the equivariant Euler characteristic $\chi(E) \in \ZZ[\![e^\lambda]\!][e^{-\lambda}]_{\rm rat}$ as:
\begin{equation} 
\label{eq:Euler_char}
\chi(E) := \sum_{i=0}^{\dim X} 
(-1)^i \ch^T\big(H^i(X,E)\big) 
\end{equation} 

Let $R_T = H^\bullet_T({\rm pt},\CC)$, and let $S_T$ denote the localization of $R_T$ with respect to the set of non-zero homogeneous elements.
We expect that properties (P1)~and~(P2) are sufficient to imply the following 
equivariant Hirzebruch--Riemann--Roch (HRR) formula:  
\begin{equation}
\label{eq:equiv_HRR} 
\chi(E) = 
\int_{IX} \tch(E) \cup \tTd_X. 
\end{equation} 
This identity should be interpreted with care. 
The right-hand side is an equivariant integral (defined via the localization 
formula) of an element 
of $H^{\bullet\bullet}_T(IX)$, and lies in a completion $\hS_T$ of $S_T$:  
\[
\hS_T := \left\{
\sum_{n\in \ZZ} a_n : \text{$a_n \in S_T^n$, there exists $n_0\in \ZZ$ such that $a_n = 0$ for all $n< n_0$} \right\}  
\]
where $S_T^n$ denotes the degree $n$ graded 
component of $S_T$. 
As we discussed above, the left-hand side of \eqref{eq:equiv_HRR} 
lies in $\ZZ[\![e^\lambda]\!][e^{-\lambda}]_{\rm rat}$ 
and is given by a rational function $f(e^{\lambda_1},
\dots, e^{\lambda_m})$. We take the Laurent expansion 
of $g(t) = f(e^{t\lambda_1},\dots,e^{t\lambda_m})$ at 
$t=0$ and obtain an expression $g(t)=\sum_{n\ge n_0} g_n t^n$ 
with $g_n \in S_T^n$. The HRR formula \eqref{eq:equiv_HRR} 
claims that the element $\sum_{n\ge n_0} g_n \in \hS_T$ thus obtained 
is equal to the right-hand side of~\eqref{eq:equiv_HRR}. Note that we have the following inclusions of 
rings: 
\[
\ZZ[\![e^{\lambda}]\!][e^{-\lambda}] \supset 
\ZZ[\![e^{\lambda}]\!][e^{-\lambda}]_{\rm rat} 
\lhook\joinrel\relbar\joinrel\rightarrow \hS_T.  
\]

Non-equivariant versions of the HRR formula~\eqref{eq:equiv_HRR} for orbifolds and Deligne--Mumford stacks have been established by Kawasaki~\cite{Kawasaki} and Toen~\cite{Toen}.  (In the non-equivariant case, $X$ has to be compact so that both sides of \eqref{eq:equiv_HRR} are well-defined.)  The equivariant index theorem has been studied by many authors (see e.g.~\cite{Berline-Getzler-Vergne, Edidin-Graham, Kock} and references therein) and the formula \eqref{eq:equiv_HRR} is known to hold (at least) for compact smooth manifolds~\cite{Duistermaat, Edidin-Graham}, compact orbifolds~\cite{Vergne}, and proper Deligne--Mumford stacks~\cite{Edidin}.  We could not, however, find a reference for the formula~\eqref{eq:equiv_HRR} for non-proper Deligne--Mumford stacks. In~\S\ref{sec:localization}, we establish \eqref{eq:equiv_HRR} for toric Deligne--Mumford stacks, using localization in equivariant $K$-theory.

\begin{example} 
Consider $\CC^2$ with the diagonal $\CC^\times$-action. 
The Euler characteristic 
of the structure sheaf is: 
\[
\ch^{\CC^\times}(H^0(\CC^2,\cO))  = \sum_{n=0}^\infty (n+1) 
e^{n\lambda}.
\]
On the other hand, 
\[
\int_{\CC^2} \Td_{\CC^2}^{\CC^\times} = \int_{\CC^2} \frac{(-\lambda)^2}
{(1-e^{\lambda})^2} = \frac{1}{(1-e^{\lambda})^2} 
= \frac{1}{\lambda^2} - \frac{1}{\lambda} + \frac{5}{12} 
- \frac{1}{12} \lambda + \frac{1}{240}\lambda^2 + \cdots.  
\]
The two quantities match. 
If we consider instead the anti-diagonal $\CC^\times$-action 
$(x,y) \mapsto (s^{-1}x, s y)$ on $\CC^2$, 
the Euler characteristic is ill-defined since each weight 
subspace is infinite dimensional; this action does not satisfy our assumptions.  
\end{example} 

\section{Localization in Equivariant $K$-Theory}
\label{sec:localization}

In this section we prove the Localization Theorem for the $T$-equivariant $K$-theory of toric Deligne--Mumford stacks, using methods and results of Thomason~\cite{Thomason:1,Thomason:2,Thomason:3}.  We then deduce the $T$-equivariant Hirzebruch--Rieman--Roch formula~\eqref{eq:equiv_HRR}.

\subsection{Toric Deligne--Mumford Stacks as GIT Quotients}
\label{sec:GIT}
The definition of toric Deligne-Mumford stacks is given by \cite{Borisov--Chen--Smith}, and we mainly follow the notations in \S 4 of \cite{Coates--Iritani--Jiang}. 
Let $K = (\Cstar)^r$.  Let $\LL= \Hom(\Cstar,K)$ denote the cocharacter lattice of $K$, so that $\LL^\vee = 
\Hom(K,\Cstar)$ is the lattice of characters, and fix characters $D_1,\ldots,D_m \in \LL^\vee$.  This choice of characters defines a map from $K$ 
to the torus $T = (\Cstar)^m$, and hence  an action of $K$ on $\CC^m$.  

\begin{notation}
  \label{not:cones}
  For a subset $I$ of $\{1, 2, \ldots, m\}$, write $\overline{I}$ 
  for the complement of $I$, and set:
  \begin{align*}
    \Cone_I &= \big\{ \textstyle\sum_{i \in I} a_i D_i : \text{$a_i \in
      \RR$, $a_i > 0$} \big\} \subset
    \LL^\vee \otimes \RR \\
    (\Cstar)^I \times \CC^{\overline{I}} 
    & = \big \{ (z_1,\dots,z_m) : z_i \neq 0 \text{ for } i\in I 
\big\} \subset \CC^m
   \end{align*} 
   We set $\Cone_\emptyset = \{0\}$. 
\end{notation} 

\begin{definition} 
\label{def:toricstack} 
Consider now a \emph{stability condition} 
$\omega \in \LL^\vee \otimes \RR$, and set:
\begin{align*}  
    \cA_\omega &= 
    \big\{ I \subset \{1,2,\ldots,m\} : \omega \in \Cone_I \big\} \\
    U_\omega &=\bigcup_{I\in \cA_\omega} 
(\Cstar)^I \times \CC^{\overline{I}} \\
    X_\omega &= \big[U_\omega \big/ K \big]
\end{align*}
The square brackets here indicate that $X_\omega$ is the stack quotient 
of $U_\omega$ (which is $K$-invariant) by $K$.  
We call $X_\omega$ the \emph{toric stack 
associated to the GIT data} $(K;\LL;D_1,\ldots,D_m; \omega)$. 
Elements of $\cA_\omega$ are called \emph{anticones}\footnote{This terminology is explained in~\cite[\S4.2]{Coates--Iritani--Jiang}.}.
\end{definition}
\begin{example}
 \label{ex:conifold1}
Let us give a simple example, purely to demystify all this notation. Set
$r=1$, so that $K=\Cstar$ and $\LL^\vee=\ZZ$. Now set $m=4$, and choose
$D_1=D_2=1$ and $D_3=D_4=-1$. This means we are considering a GIT quotient
of $\CC^4$ under a diagonal $\Cstar$ action having weights $(1,1,-1,-1)$. We
need to choose a stability condition $\omega$  in $\LL^\vee\otimes \RR=\RR$,
and let us choose $\omega$ to be positive. Then $\cA_\omega$ consists of all
subsets $I\subset \{1,..,4\}$ such that either $1\in I$ or $2\in I$. Hence
$U_\omega$ contains the two open sets:
$$(\Cstar)^{\{1\}}\times \CC^{\{2,3,4\}} = \{z_1\neq
0\}\hspace{2cm}\mbox{and}\hspace{2cm} (\Cstar)^{\{2\}}\times \CC^{\{1,3,4\}}
= \{z_2\neq 0\}$$
In fact these two open sets cover $U_\omega$  -- any other $I\in \cA_\omega$
corresponds to a subset of at least one of these two -- so $U_\omega$ is the
subset $\{ (z_1, z_2)\neq (0,0)\}\subset \CC^4$.  Thus $X_\omega$ is the
total space of the rank 2 vector bundle $\cO(-1)^{\oplus 2}$ over $\PP^1$.
\end{example}

Unless otherwise stated, we will consider only GIT data that satisfy:
\begin{equation}
  \label{eq:assumption}
  \parbox{0.8\textwidth}{
    \begin{itemize}
    \item $\{1,2,\ldots,m\} \in \cA_\omega$;
    \item for each $I \in \cA_\omega$, the set $\{D_i : i \in I\}$ spans $\LL^\vee \otimes \RR$ over $\RR$.
    \end{itemize}
  }
\end{equation}
The first condition here ensures that $X_\omega$ is non-empty; 
the second ensures that $X_\omega$ is a Deligne--Mumford stack. 
These conditions imply that $\cA_\omega$ is closed under 
enlargement of sets, so that if $I\in \cA_\omega$ and 
$I\subset J$ then $J \in \cA_\omega$.   

Fixed points of the $T$-action on $X_\omega$ are 
in one-to-one correspondence with minimal anticones, that is, with $\delta \in \cA_\omega$ such that $|\delta| = r$.  
A minimal anticone $\delta$ corresponds to the $T$-fixed point:
\[
\big[\{(z_1,\ldots,z_n) \in U_\omega : \text{$z_i = 0$ if $i \not \in
  \delta$}\} \big/ K \big] = \big[ (\Cstar)^\delta \big/ K \big] 
\]
Let $\Fix_\omega$ denote the set of minimal anticones for $X_\omega$.

\subsection{The Localization Theorem}

We now state and prove our Localization Theorem.

\begin{theorem}
\label{thm:localization}
Let $X_\omega = \big[U_\omega \big/ K \big]$ be 
a toric Deligne--Mumford stack as above.
Recall that the torus $T$ acts (ineffectively) on $X_\omega$.  
Given $\delta \in \Fix_\omega$, write $x_\delta$ for 
the corresponding $T$-fixed point of $X_\omega$, 
so that $x_\delta \cong B G_\delta$ 
where $G_\delta$ is the isotropy subgroup of $x_\delta$. 
Let 
$i_\delta \colon x_\delta \to X_\omega$ denote the 
inclusion and let $N_\delta$ denote the normal bundle 
to $i_\delta$. 
Let $\ZZ[T] = K_T^0({\rm pt})$ denote the ring of 
regular functions (over $\ZZ$) on $T$ and let $\Frac\ZZ[T]$ 
denote the field of fractions. 
Then for $\alpha \in K_T^0(X_\omega)$, we have:  
\begin{align*}
\alpha = \sum_{\delta\in \Fix_\omega} (i_\delta)_\star 
\left( 
\frac{i_\delta^\star \alpha}{\lambda_{-1} N_\delta^\vee}
\right) 
&& \text{in $K_T^0(X_\omega)\otimes_{\ZZ[T]} \Frac(\ZZ[T])$}
\end{align*}
where $\lambda_{-1} N_\delta^\vee  
:= \sum_{i=0}^{\dim X_\omega} (-1)^i \bigwedge^i N_\delta^\vee$ 
is invertible in $K_T^0(x_\delta)\otimes_{\ZZ[T]} 
\Frac(\ZZ[T])$. 
\end{theorem}

\begin{proof}
We have that $K^0_T(X_\omega) = K^0_{T\times K}(U_\omega)$, 
where the action of $(t,k) \in T \times K$ on $U_\omega$ is given by 
the action of $t k^{-1} \in T$ on $U_\omega$.  
As a module over $K_{T\times K}^0({\rm pt}) = \ZZ[T\times K]$, 
$K^0_{T\times K}(U_\omega)$ is supported\footnote{The equivariant $K$-group $K^0_{T\times K}(U_\omega)$ is a module over $K_{T\times K}^0({\rm pt}) = \ZZ[T\times K]$, and hence defines a sheaf on $\Spec \ZZ[T\times K]$.  The support of $K^0_{T\times K}(U_\omega)$ means the support of this sheaf.}
 on 
the set of points $(t,k)\in T\times K$ such that 
$(t,k)$ has a fixed point in $U_\omega$ 
\cite[Th\'{e}or\`{e}me 2.1]{Thomason:3}. 
Therefore the support of $K_{T\times K}^0(U_\omega)$ 
is the union $\bigcup_{\delta \in \Fix_\omega} T_\delta$ 
of subtori $T_\delta$ defined by 
\begin{equation}
\label{eq:T_delta}
T_\delta = \{(t,k) \in T\times K : \pi_\delta(t) = \pi_\delta(k) \}. 
\end{equation} 
Here $\pi_\delta \colon T = (\Cstar)^m \to (\Cstar)^\delta$ is the natural 
projection. 
Note that $T_\delta$ fixes the locus 
$(\Cstar)^\delta \subset U_\omega$ 
corresponding to the fixed point $x_\delta$. 
The torus $T_\delta$ is connected and 
the natural projection $T_\delta \to T$ is a finite covering 
with Galois group $G_\delta$. 
Therefore the localization $K_{T\times K}^0(U_\omega) 
\otimes_{\ZZ[T]} \Frac(\ZZ[T])$ is supported 
on finitely many points, which are the generic points $\xi_\delta$ 
of $T_\delta$. 
On the other hand, the stalk of $K_{T\times K}^0(U_\omega)$ 
at $\xi_\delta$ is given by the isomorphism 
\cite[Th\'{e}or\`{e}me 2.1]{Thomason:3}: 
\begin{equation} 
\label{eq:i_delta_star} 
(i_{\delta})_\star \colon K_{T\times K}^0((\Cstar)^\delta)_{\xi_\delta} 
\xrightarrow{\cong}
K_{T\times K}^0(U_\omega)_{\xi_\delta}
\end{equation} 
The localization $K^0_{T\times K}(U_\omega) \otimes_{\ZZ[T]} 
\Frac(\ZZ[T])$ is the direct sum of these stalks. 
For the same reason, we have:
\[
K_T^0(x_{\delta}) \otimes_{\ZZ[T]} \Frac(\ZZ[T]) = 
K_{T\times K}^0((\Cstar)^\delta)) 
\otimes_{\ZZ[T]}\Frac (\ZZ[T]) 
= 
K_{T\times K}^0((\Cstar)^\delta)_{\xi_\delta}
\] 
The inverse to \eqref{eq:i_delta_star} is given by 
$(\lambda_{-1} N_\delta^\vee)^{-1}\cdot i_\delta^\star(-)$ by~\cite[Lemma 3.3]{Thomason:3}. The conclusion follows. 
We remark that 
$\lambda_{-1} N_\delta^\vee$ 
is invertible in $K^0_{T\times K}((\Cstar)^\delta)_{\xi_\delta}$ 
by \cite[Lemma 3.2]{Thomason:3}. 
\end{proof}

\begin{corollary}
\label{cor:fixed_point_formula} 
Let the notation be as in Theorem \ref{thm:localization}. 
For $\alpha\in K^0_T(X_\omega)$, we have 
  \[
  \chi(\alpha) = \sum_{\delta \in \Fix_\omega} 
\chi\left(\frac{i_\delta^\star \alpha}{\lambda_{-1} N_\delta^\vee}\right)
  \]
where $\chi(-)$ denotes the $T$-equivariant Euler characteristic 
given in \eqref{eq:Euler_char}. 
\end{corollary}

\begin{proof}
The discussion in \S \ref{sec:HRR} shows that $\chi$ defines 
a $\ZZ[T]$-linear map 
\[
K_T^0(X_\omega) \to \Frac(\ZZ[T])  
\]
which, by extension of scalars, gives 
$K_T^0(X_\omega)\otimes_{\ZZ[T]} \Frac(\ZZ[T]) \to \Frac(\ZZ[T])$. 
Corollary \ref{cor:fixed_point_formula} is thus an immediate consequence 
of Theorem \ref{thm:localization}. 
\end{proof}

\begin{corollary} 
\label{cor:equiv_HRR} 
The $T$-equivariant Hirzebruch--Riemann--Roch formula 
\eqref{eq:equiv_HRR} 
holds when $X$ is a toric Deligne--Mumford stack with semi-projective coarse moduli space and the torus-fixed set $X^T$ is non-empty.
\end{corollary} 

\begin{proof}
We compute the right-hand side of the HRR formula \eqref{eq:equiv_HRR} 
using localization in equivariant cohomology, and match it with 
the fixed point formula in Corollary~\ref{cor:fixed_point_formula}. 
Recall the $(T \times K)$-action in the proof of Theorem \ref{thm:localization}.
It suffices to show that 
\[
\chi\left(\frac{V}{\lambda_{-1} N_\delta}\right) = 
\frac{1}{|G_\delta|} 
\sum_{g \in G_\delta}\frac{\tch(V)_g \tTd(N_\delta)_g}{e_T(N_{\delta,g})}  
\]
for a $(T\times K)$-representation $V$. 
Here we regard $V$ as a $(T\times K)$-equivariant vector bundle 
on $(\Cstar)^\delta$, which is the same thing as 
a $T$-equivariant vector bundle on $x_\delta = [(\Cstar)^\delta/K]$.  
The index $g\in G_\delta$ parametrizes connected components of $IBG_\delta$, $\tch(\cdot)_g$ and $\tTd(\cdot)_g$ denote the components of the $T$-equivariant orbifold Chern character and $T$-equivariant orbifold Todd class along the component of $IBG_\delta$ indexed by $g$,
and $N_{\delta,g}$ is the $g$-fixed subbundle of $N_\delta$.  

Consider the subgroup $T_\delta$ of $T\times K$ in \eqref{eq:T_delta}. 
This is the stabilizer of the $(T\times K)$-action on $(\Cstar)^\delta$ and 
fits into the exact sequence: 
\[
\xymatrix{
  1 \ar[r] & G_\delta \ar[r] &  T_\delta \ar[r] & T \ar[r] & 1
}
\]
A $(T\times K)$-representation $W$ can be viewed as 
a $T_\delta$-representation and the $G_\delta$-invariant
part $W^{G_\delta}$ gives a $T$-representation. 
The Euler characteristic of $W$, as a $T$-equivariant 
vector bundle on $x_\delta$, is then given by the 
$T$-character of $W^{G_\delta}$: 
\[
\chi(W) = \ch^T(W^{G_\delta}) = 
\frac{1}{|G_\delta|}\sum_{g\in G_\delta} \Tr(g e^{\lambda} \colon W) 
\]
where $\lambda \in \Lie(T)$ and $g e^{\lambda}$ gives an element 
of $T_\delta$. 
On the other hand, we have 
\begin{align*} 
\Tr(g e^{\lambda} \colon V) 
& =  \tch(V)_g 
\\
\Tr(g e^{\lambda} \colon \lambda_{-1} N_\delta^\vee) 
& = \frac{e_T(N_{\delta,g})}{\tTd(N_\delta)_g} 
\end{align*} 
by the definition of $\tch$ and $\tTd$ in \S\ref{sec:HRR}. 
The conclusion follows from the fact that 
$\Tr(ge^{\lambda} \colon -)$ preserves the product. 
\end{proof}

\section{Birational Transformations from Variation of GIT}
\label{sec:GIT_variation}

In this section we consider crepant birational transformations $\varphi \colon X_+ \dashrightarrow X_-$ between toric Deligne--Mumford stacks which arise from a variation of GIT quotient.  We construct a $K$-equivalence:
\begin{equation}
  \label{eq:K-equivalence}
  \begin{aligned}
    \xymatrix{
      & \widetilde{X} \ar[ld]_{f_+} \ar[rd]^{f_-} \\
      X_+  \ar@{-->}[rr]^{\varphi}  && X_- }
  \end{aligned}
\end{equation}
canonically associated to $\varphi$, and show that this too arises from a variation of GIT quotient. 

\bigskip

Recall that our GIT data in \S \ref{sec:GIT} 
consist of a torus $K \cong (\Cstar)^r$, 
the lattice $\LL = \Hom(\Cstar,K)$ of 
$\Cstar$-subgroups of $K$, 
and characters $D_1,\ldots,D_m \in \LL^\vee$.  
A choice of stability condition 
$\omega \in \LL^\vee \otimes \RR$ satisfying 
\eqref{eq:assumption} determines a toric 
Deligne--Mumford stack $X_\omega = \big[U_\omega/K\big]$.  
The space $\LL^\vee \otimes \RR$ of stability conditions is divided into 
chambers by the closures of the sets $\Cone_I$, $|I| = r-1$, 
and the Deligne--Mumford stack $X_\omega$ depends on $\omega$ 
only via the chamber containing $\omega$.  
For any stability condition $\omega$ satisfying \eqref{eq:assumption}, 
the set $U_\omega$ contains the big torus $T=(\Cstar)^m$, 
and thus for any two such stability conditions $\omega_1$,~$\omega_2$ 
there is a canonical birational map 
$X_{\omega_1} \dashrightarrow X_{\omega_2}$, 
induced by the identity transformation between 
$T/K \subset X_{\omega_1}$ and 
$T/K \subset X_{\omega_2}$.  

Consider now a birational transformation $X_+ \dashrightarrow X_-$ arising from a single wall-crossing in the space of stability conditions, as follows.
Let $C_+$,~$C_-$ be chambers 
in $\LL^\vee \otimes \RR$ that are separated by a hyperplane wall $W$, 
so that $W \cap \overline{C_+}$ is a facet of $\overline{C_+}$, 
$W \cap \overline{C_-}$ is a facet of $\overline{C_-}$, 
and $W \cap \overline{C_+} = W \cap \overline{C_-}$.  
Choose stability conditions $\omega_+ \in C_+$, $\omega_- \in C_-$ 
satisfying \eqref{eq:assumption} and set $U_+ := U_{\omega_+}$, $U_- := U_{\omega_-}$,
$X_+ := X_{\omega_+}$, $X_- := X_{\omega_-}$, and:
\begin{align*}
  & \cA_\pm := \cA_{\omega_{\pm}} = 
\big\{ I \subset \{1,2,\ldots,m\} : 
\omega_\pm \in \Cone_I \big\} 
\end{align*}
Then $C_\pm = \bigcap_{I\in \cA_\pm} \Cone_I$. 
Let $\varphi \colon X_+ \dashrightarrow X_-$ be 
the birational transformation induced by the toric wall-crossing from $C_+$ to $C_-$
and suppose that $\sum_{i=1}^m D_i \in W$: as we will see 
below this amounts to requiring that $\varphi$ is crepant. 
Let $e \in \LL$ denote the primitive lattice vector in $W^\perp$ 
such that $e$ is positive on $C_+$ and negative on $C_-$. 

\begin{example}\label{ex:conifold2}
Recall our Example \ref{ex:conifold1} where we quotiented $\CC^4$ by $\Cstar$ with weights $(1,1,-1,-1)$. Here there are exactly two chambers in $\LL^\vee\otimes \RR=\RR$, namely $C_+=\RR_+$ and $C_-=\RR_-$, and they are separated by the wall $W=\{0\}$.  Note that $\sum D_i = 0$ which does indeed lie in $W$, and that $e$ just the vector $1\in \RR$.

 Previously we chose a postive $\omega$ so we constructed $X_+$; if we now choose an $\omega_-\in C_-$ then we get that $U_- = \{(z_3,z_4)\neq (0,0)\}\subset \CC^4$, and $X_-$ is again the total space of $\cO(-1)^{\oplus 2}$ over (a different) $\PP^1$. The birational map $\varphi \colon X_+ \dashrightarrow X_-$ is the Atiyah flop. 
\end{example}

Choose $\omega_0$ from the relative interior of $W\cap \overline{C_+} 
= W \cap \overline{C_-}$. The stability condition $\omega_0$ 
does not satisfy our assumption \eqref{eq:assumption} on GIT data, but we can 
still consider 
\begin{align*} 
\cA_0 & := \cA_{\omega_0} = 
\left\{ I \subset \{1,\dots,m\} : 
\omega_0 \in \Cone_I \right\} 
\end{align*} 
and the corresponding toric Artin stack $X_0 := X_{\omega_0}
=[U_{\omega_0}/K]$ 
as given in Definition \ref{def:toricstack}. 
Here $X_0$ is not Deligne--Mumford, as 
the $\Cstar$-subgroup of $K$ 
corresponding to $e\in \LL$ (the defining equation 
of the wall $W$) has a fixed point in $U_0 := U_{\omega_0}$. 
The stack $X_0$ contains both $X_+$ and $X_-$ as open substacks 
and the canonical line bundles of $X_{+}$ and $X_-$ 
are the restrictions of the same line bundle $L_0\to X_0$ 
given by the character ${-\sum_{i=1}^m} D_i$ of $K$. 
The condition $\sum_{i=1}^m D_i\in W$ 
ensures that $L_0$ comes from 
a $\QQ$-Cartier divisor on the underlying singular 
toric variety 
$\overline{X}_0 = \CC^m/\!\!/_{\omega_0} K$.  There are canonical blow-down maps $g_\pm \colon X_\pm \to \overline{X}_0$, and $K_{X_\pm}=g_\pm^\star L_0$.  The maps $g_\pm$ will combine with diagram \eqref{eq:K-equivalence} to give a commutative diagram:
\[
\xymatrix{
  & \tX \ar[rd]^{f_-} \ar[ld]_{f_+} & \\ 
  X_+ \ar[rd]_{g_+} \ar@{-->}^{\varphi}[rr] &  & X_- \ar[ld]^{g_-} \\ 
  & \overline{X}_0 & 
}
\]
This shows that $f_+^\star(K_{X_+})$ and 
$f_-^\star(K_{X_-})$ coincide, since they 
are the pull-backs of the same $\QQ$-Cartier divisor on $\overline{X}_0$. 
The equality  $f_+^\star(K_{X_+}) = f_-^\star(K_{X_-})$  is what is meant by the birational map $\varphi$ being
\emph{crepant}, and by the diagram~\eqref{eq:K-equivalence} being a \emph{$K$-equivalence}. 

\begin{example}\label{ex:conifold3}
Let us continue with our Example \ref{ex:conifold2}. The stability condition $\omega_0$ can only be $0\in \RR$. This means that the empty set $\emptyset$ is an element of $\cA_0$ (since $\Cone_\emptyset = \{0\}$ by definition), hence $U_0$ is the whole of $\CC^4$ and $X_0$ is the Artin stack $[\CC^4/\Cstar]$. The underlying singular variety $\overline{X}_0$ is the 3-fold ordinary double point.
\end{example}

It remains to construct diagram \eqref{eq:K-equivalence}.  Consider the action of $K \times \Cstar$ on $\CC^{m+1}$ defined by the characters 
$\tD_1,\ldots,\tD_{m+1}$ of $K \times \Cstar$, where:
\[
\tD_j = 
\begin{cases}
  D_j \oplus 0 & \text{if $j < m+1$ and $D_j \cdot e \leq 0$} \\
  D_j \oplus ({-D_j} \cdot e) & \text{if $j < m+1$ and $D_j \cdot e > 0$} \\
  0 \oplus 1 & \text{if $j=m+1$}
\end{cases}
\]
Consider the chambers $\widetilde{C}_+$,~$\widetilde{C}_-$, 
and~$\widetilde{C}$ in $(\LL \oplus \ZZ)^\vee \otimes \RR$ 
that contain, respectively, the stability conditions
\begin{align*}
  \tomega_+ = (\omega_+,1) &&
  \tomega_- = (\omega_-,1) && \text{and} && 
  \tomega = (\omega_0, - \varepsilon)
\end{align*}
where $\varepsilon$ is a very small positive real number.  
Let $\tX$ denote the toric Deligne--Mumford stack 
defined by the stability condition $\tomega$.  Lemma~6.16 in \cite{Coates--Iritani--Jiang} gives that:
\begin{enumerate}
\item The toric Deligne--Mumford stack corresponding 
  to the chamber $\widetilde{C}_+$ is $X_+$.
\item The toric Deligne--Mumford stack corresponding 
  to the chamber $\widetilde{C}_-$ is $X_-$.
\item There is a commutative diagram as in \eqref{eq:K-equivalence}, where:
  \begin{itemize}
  \item $f_+ \colon \tX \to X_+$ is a toric blow-up, 
    arising from the wall-crossing from $\widetilde{C}$ to $\widetilde{C}_+$; 
    and
  \item $f_- \colon \tX \to X_-$ is a toric blow-up, 
    arising from the wall-crossing from $\widetilde{C}$ to $\widetilde{C}_-$.
  \end{itemize}
\end{enumerate}

\begin{example}\label{ex:conifold4}
Let us spell out this construction for the Atiyah flop (that is, we continue our Example~\ref{ex:conifold3}). We consider an action of $(\Cstar)^2$ on $\CC^5$ with weight matrix
$$\begin{pmatrix} 1 & 1 & -1 & -1 & 0 \\ -1 & -1 & 0 & 0 & 1 \end{pmatrix} $$
(The columns of this matrix are the characters $\tD_1, ..., \tD_5$.) The space of stability conditions $\RR^2=\{(s_1, s_2)\}$ is partitioned into three chambers:
$$\widetilde{C}_+ = \{s_1>0, s_1+s_2>0\},\hspace{1cm} 
\widetilde{C}_- = \{s_1<0, s_2>0\},\hspace{1cm} \mbox{and}\hspace{1cm}
\widetilde{C} = \{s_1+s_2<0, s_2<0\} $$
The walls are the rays spanned by the characters $\tD_i$. 

If we choose a stability condition $\omega$ lying in $\widetilde{C}_+$ or $\widetilde{C}_-$ (such as $\omega=\tomega_+$ or $\omega=\tomega_-$) then any anticone $I\in \cA_\omega$ has to have $5\in I$; consequently the semi-stable locus $U_\omega$ is contained in the open set $\{z_5\neq 0\}\subset \CC^5$. However, the stack $[\{z_5\neq 0\} / (\Cstar)^2]$ is canonically equivalent to the stack $X_0=[\CC^4/\Cstar]$, so for these stability conditions the GIT problem reduces to the previous one. This is why stability conditions in $\widetilde{C}_+$ produce $X_+$, and stability conditions in $\widetilde{C}_-$ produce $X_-$.

Now consider the stability condition $\tomega$ from the chamber $\widetilde{C}$. The anticones $I\in \cA_\tomega$ are the subsets of $\{1,...,5\}$ such that $I\cap \{1,2\}\neq \emptyset$ and $I\cap \{3,4\}\neq \emptyset$. Consequently the semi-stable locus for $\tomega$ is the open set:
$$U_{\tomega}= \left\{  (z_1,z_2)\neq(0,0), \; (z_3, z_4)\neq (0,0)    \right\} \; \subset \CC^5 $$
Then $\tilde{X}$ is the total space of the line-bundle $\cO(-1,-1)$ over $\PP^1\times\PP^1$. This is the common blow-up of $X_+$ and $X_-$.
\end{example}

\section{The Fourier--Mukai Functor is a Derived Equivalence} 
\label{sec:derived_equivalence}

Let $\varphi \colon X_+ \dashrightarrow X_-$ be a crepant birational transformation between toric Deligne--Mumford stacks which arises from a toric wall-crossing, and let:
\begin{equation}
  \label{eq:K-equivalence_again}
  \begin{aligned}
    \xymatrix{
      & \widetilde{X} \ar[ld]_{f_+} \ar[rd]^{f_-} \\
      X_+  \ar@{-->}[rr]^{\varphi}  && X_- }
  \end{aligned}
\end{equation}
be the $K$-equivalence constructed in \S\ref{sec:GIT_variation}.  Recall that $X_+ = \big[U_{\omega_+}/K\big]$, $X_- = \big[U_{\omega_-}/K\big]$ where $U_{\omega_\pm}$ are open subsets of $\CC^m$ and 
$K = (\Cstar)^r$ acts on $\CC^m$ via a homomorphism 
$K \to T$ with finite kernel. 
Set $Q = \big[T/K\big]$, so that $X_+$ and $X_-$ carry effective actions of $Q$.  The maps $f_\pm$ in \eqref{eq:K-equivalence_again} are $Q$-equivariant.  In this section we show that the Fourier--Mukai functor:
\begin{align*}
  \FM \colon D^b_Q(X_-) \to D^b_Q(X_+) && \FM := (f_+)_\star (f_-)^\star
\end{align*}
is an equivalence of categories.  This generalizes a theorem due to Kawamata \cite[Theorem 4.2]{Kawamata}, by considering the $Q$-equivariant, rather than the non-equivariant, derived category.

To prove that the Fourier--Mukai transform gives an equivarant derived equivalence, we will use the theory developed by Halpern-Leistner~\cite{Halpern-Leistner} and Ballard--Favero--Katzarkov~\cite{Ballard--Favero--Katzarkov} which relates derived categories to variation of GIT. Note that the $Q$-equivariant derived category of $X_\pm$ is just the derived category of the stack $\big[X_\pm/Q\big] = \big[U_{\omega_\pm}/T\big]$, and that $\big[U_\pm/T\big]$ both sit as open substacks of $\big[\CC^m/T\big]$. The work of Halpern-Leistner and Ballard--Favero--Katzarkov allows us to find a (non-unique) subcategory:
\[
\mathbf{G} \subset \Dstack{\CC^m}{T} 
\]
which is equivalent, under the restriction functors, to both $D^b_Q(X_-)$ and $D^b_Q(X_+)$. By inverting the first equivalence we get an equivalence:
\[
\mathbb{GR} \colon D^b_Q(X_-) \stackrel{\sim}{\longrightarrow} D^b_Q(X_+)
\]
The notation $\mathbb{GR}$ here refers to the `grade-restriction rules' which define the subcategory $\mathbf{G}$. We will show that $\mathbb{GR}$ and $\FM$ are the same functor, hence proving that the $\FM$ is an equivalence. 

\begin{remark}
  The result that $\mathbb{GR}=\FM$ is stated in \cite[\S3.1]{Halpern-Leistner--Shipman}, and a sketch proof is given. We did not find the sketch entirely satisfactory, and so give a complete proof here.  (Also Halpern-Leistner--Shipman treat only the non-equivariant case, but this is a minor point.)
\end{remark}

\subsection{Grade-Restriction Rules}\label{sect:GRR}

The theory we need was developed by Halpern-Leistner~\cite{Halpern-Leistner} and Ballard--Favero--Katzarkov~\cite{Ballard--Favero--Katzarkov} independently; we will quote the former.  We consider only smooth spaces acted on by tori, this simplifies the theory considerably.  Let $M$ be a smooth variety carrying an action of a torus $G$. A \emph{Kempf--Ness stratum} (henceforth \emph{KN-stratum}) consists of the following data:
\begin{itemize}
\item A 1-parameter subgroup $\lambda\subset G$.  \item A connected component $Z$ of the fixed locus $M^\lambda$. We let $i_Z: Z\hookrightarrow M$ denote the inclusion.  \item The associated \emph{blade}:
  \[
  S = \left\{ y\in M;\;\; \lim_{t\to \infty} \lambda(t) (y) \in Z \right\}
  \]
  We require that $S$ is closed in $M$.
\end{itemize}
Both $Z$ and $S$ are automatically smooth, and a theorem of Bia\l ynicki--Birula implies that $S$ is a locally trivial bundle of affine spaces over $Z$.  The fixed component $Z$ is automatically closed in $M$, but $S$ need not be; thus the requirement that $S$ be closed in $M$ is non-trivial. To a KN-stratum we associate the numerical invariant:
\[
\eta := \mbox{weight}_\lambda \left( \det(N_{S/M})|_Z\right)
\]
From the definition of $S$ we have that $\eta$ is a non-negative integer. Now pick any integer $k$, and define the subcategory
\[
\mathbf{G}_k \subset \Dstack{M}{G}
\]
to be the full subcategory consisting of objects $\cE$ that obey the following \emph{grade-restriction rule}:
\begin{equation}\label{eq:GRR}
\mbox{the homology sheaves of } Ri_Z^\star\cE \mbox{ have }\lambda \mbox{ weights lying in the interval } [k, k+\eta). 
\end{equation}
The main result of~\cite{Halpern-Leistner}, Theorem~3.35 there, is that for any $k$ the restriction functor gives an equivalence:
\[
\xymatrix{
  \mathbf{G}_k \ar[r]^-\sim & \Dstack{M \setminus S}{G} 
}
\]
A \emph{KN-stratification} is a sequence $(\lambda_0, Z_0, S_0), ..., (\lambda_n, Z_n, S_n)$ such that each triple $(\lambda_i, Z_i, S_i)$ is a KN-stratum in the space $M \setminus \bigcup_{j<i} S_j$. If we pick an integer $k_i$ for each stratum then we can define a subcategory:
\[
\mathbf{G}_{k_\bullet} \subset \Dstack{M}{G}
\]
by imposing a grade-restriction rule on each locally closed subvariety $Z_i\subset M$. By recursively applying the previous result we have~\cite[Theorem 2.10]{Halpern-Leistner} that $\mathbf{G}_{k_\bullet}$ is equivalent to the derived category of:
\[
\Big[\Big(M \setminus \bigcup_i S\Big)\big/G\Big]
\]
If $M$ is semi-projective and $M^{ss}$ is the semi-stable locus for some stability condition, then Kempf and Ness showed that we can construct a KN-stratification with $M \setminus \bigcup_i S_i = M^{ss}$.  Thus the subcategory $\mathbf{G}_{k_\bullet}$ provides a way to lift the derived category of the GIT quotient $\big[M^{ss}/G\big]$ into the derived category of the ambient Artin stack $\big[M / G\big]$.

Next we explain how to apply this theory to find the derived equivalence
\[
\xymatrix{
  \mathbb{GR} \colon D^b_Q(X_-) \ar[r]^-\sim &  D^b_Q(X_+)
}
\]
following~\cite[\S4.1]{Halpern-Leistner}. In~\S\ref{sec:GIT_variation} above we introduced open subsets of $\CC^m$
\begin{align*}
U_+ = \bigcup_{I\in \cA_+} (\Cstar)^I \times \CC^{\overline{I}} && 
U_0 = \bigcup_{I\in \cA_0} (\Cstar)^I \times \CC^{\overline{I}} &&
U_- = \bigcup_{I \in \cA_-} (\Cstar)^I \times \CC^{\overline{I}}
\end{align*}
with $X_+ = \big[ U_+/ K \big]$ and $X_- = \big[ U_-/ K \big]$. The set $U_0$ is the semi-stable locus for a stability condition $\omega_0$ that lies on the wall $W$ between $X_+$ and $X_-$. Recall that $e$ is a primitive normal vector to $W$; this defines a 1-parameter subgroup of $K$ which `controls the wall-crossing'.   Set: 
\begin{align*} 
M_{\pm} & = \{ i \in \{1,\dots,m\} : \pm D_i\cdot e >0 \} & M_{\geq 0} &= M_0\sqcup M_+\\ 
M_0  & = \{ i\in \{1,\dots,m\} : D_i \cdot e = 0\} & M_{\leq 0} &= M_0\sqcup M_-
\end{align*} 
Our assumptions imply that both $M_+$ and $M_-$ are non-empty. 
The fixed-point locus, attracting subvariety, and repelling subvariety for $e$ are $\CC^{M_0}$, $\CC^{M_{\leq 0}}$, and $\CC^{M_{\geq 0}}$ respectively. It is clear\footnote{See e.g.~\cite[Lemma~5.2]{Coates--Iritani--Jiang}.} that $U_\pm \subset U_0$ and that:
\begin{align*}
  U_+ = U_0 \setminus \big(\CC^{M_{\leq 0}}\cap U_0\big) &&
  U_- = U_0 \setminus \big(\CC^{M_{\geq 0}} \cap U_0\big)
\end{align*}
Set:
\begin{align*}
  Z = U_0 \cap \CC^{M_0} && S_+ = U_0\cap \CC^{M_{\geq 0}} && S_-= U_0\cap \CC^{M_{\leq 0}}
\end{align*}
Both $(e, Z, S_-)$ and $(-e, Z, S_+)$ define KN-strata inside $U_0$. The numerical invariants associated to these two strata are
\begin{align*}
  \eta_+ = \sum_{i\in M_+} D_i\cdot e && \text{and} && \eta_- = {-\sum_{i\in M_-}} D_i\cdot e
\end{align*}
respectively. The crepancy condition gives $\eta_+ = \eta_-$.  Now define a full subcategory $\mathbf{G}\subset \Dstack{U_0}{T}$ consisting of objects $\cE$ such that the $e$ weights of the homology sheaves of $Ri_Z^\star \cE$ lie in the interval $[0, \eta_+)$; this is the grade-restriction rule \eqref{eq:GRR}. Then $\mathbf{G}$ is equivalent, under the restriction functor, to:
\[
\Dstack{U_0 \setminus S_-}{T} = D^b_Q(X_+) 
\]
However, this grade restriction rule is the same thing as requiring the $(-e)$ weights of the homology of $Ri_Z^\star \cE$ to lie in the interval $[-\eta_-\!+\!1, 1)$, so $\mathbf{G}$ is also equivalent to:
\[
\Dstack{U_0 \setminus S_+}{T} = D^b_Q(X_-)
\]
After inverting the latter equivalence we obtain the required equivalence $\mathbb{GR}$.

If we wish, we can pick a KN-stratification for the complement of $U_0$ in $\CC^m$ and use grade-restriction rules to lift $\Dstack{U_0}{T}$ into $\Dstack{\CC^m}{T}$, thus lifting $\mathbf{G}$ to a category defined on the larger stack. This produces the same equivalence $\mathbb{GR}$.

\subsection{Derived Categories of Blow-Ups and Variation of GIT}\label{sect:blowupsGIT}

Given a blow-up $f \colon \widetilde{X} \to X$, there are adjoint functors
\begin{align*}
  f^\star \colon D^b(X) \to D^b(\widetilde{X}) && f_\star \colon  D^b(\widetilde{X}) \to D^b(X).
\end{align*}
In this section we construct these functors using grade-restriction rules and variation of GIT, in a quite general setting.  

Suppose that $X$ is a Deligne--Mumford stack, $E$ is a vector bundle on $X$, and that $Z \subset X$ is a connected substack defined by the vanishing of a regular section $\sigma$ of $E$.  Let $\widetilde{X} := \Bl_Z X$ be the blow-up of $X$ with center $Z$.  Consider the total space of the bundle $E\oplus \cO_X$, and equip it with a $\Cstar$ action having weights $(-1, 1)$. Now consider the $\Cstar$-invariant subspace:
\[
M = \left\{ (v,z) : \text{$v \in E_x$, $z \in \CC$ such that $z v = \sigma(x)$} \right\} \; \subset E\oplus \cO_X
\]
The stack $\big[M \big/\Cstar \big]$ contains both $X$ and $\widetilde{X}$ as open substacks, and sits in a diagram:
\[
\xymatrix{
  &E \ar[dr] && \cO_{\PP(E)}(-1) \ar@{=}[r] \ar[dl]& \Bl_X E\\
  && \big[E \oplus \cO_X \big/\Cstar \big] \\
  X \ar@{=}[r] &\Gamma(\sigma) \ar[uu] \ar[dr] & & \widetilde{X} \ar@{=}[r] \ar[uu] \ar[dl] & \Bl_Z X \\
  && \big[M \big/\Cstar \big] \ar[uu]
  }
\]
where all arrows are inclusions and $\Gamma(\sigma)$ denotes the graph of $\sigma$. The fixed locus $M^{\Cstar}$ is isomorphic to $Z$, the attracting subvariety $S_-$ is isomorphic to the total space of $E \big|_Z$, and the repelling subvariety $S_+$ is isomorphic to the total space of $\cO_X \big|_Z$. Let  $U_\pm = M \setminus S_\mp$; these are the semi-stable loci for the two possible stability conditions. We have a commuting diagram:
\begin{equation}
  \label{eq:vGIT_over_X}
  \begin{aligned}
    \xymatrix{
      & & \big[M \big/ \Cstar\big] \ar[dd]^\pi \\
      X \ar@{=}[r] & \big[U_+ \big/ \Cstar\big] \ar@{^{(}->}[ur]^{i_+} \ar@{=}[dr] & & 
      \big[U_- \big/ \Cstar\big] \ar@{_{(}->}[ul]_{i_-}  \ar@{=}[r] \ar[ld] & \widetilde{X} \ar@/^/[lld]^f  \\
      & & X}
  \end{aligned}
\end{equation}
where $i_\pm$ are the inclusions and $\pi$ is induced by the vector bundle projection map $E \oplus \cO_X \to X$.  Thus the blow-up $f \colon \widetilde{X} \to X$ arises from variation of GIT, and it does so relative to $X$.

We now apply the results discussed in the previous section. Let $i_Z \colon Z \to M$ denote the inclusion. For each stability condition we have a single KN-stratum, namely $(\Cstar, Z, S_\pm)$. The numerical invariants are:
\begin{align*}
  \eta = \mbox{weight}(\cO_X|_Z) = 1 && \text{and} && 
\widetilde{\eta} = -\mbox{weight}(\det E|_Z) =  \rank E.
\end{align*}
Hence we define full subcategories:
\[
\mathbf{H} \subset \widetilde{\mathbf{H}} \subset \Dstack{M}{\Cstar}
\]
using the grade-restriction rule \eqref{eq:GRR}, where for $\mathbf{H}$ we require the weights to lie in the interval $[0,1)$ and for $\widetilde{\mathbf{H}}$ we require the weights to lie in $[0, \rank E)$. Then $\mathbf{H}$ and $\widetilde{\mathbf{H}}$ are equivalent, via the restrictions $i_+$ and $i_-$, to the derived categories of $X$ and $\tilde{X}$ respectively.

\begin{lemma}\label{lem:blowupsandGRR}\ 
\begin{enumerate}
\item The composition:
  \[
  \xymatrix{ D^b(X) \ar[r]^(.6){(i_+^\star)^{-1}} &  \mathbf{H} \ar[r]^(.4){i_-^\star}  &D^b(\tilde{X}) }
  \]
  is equal to the pull-up functor $f^\star$.
\item The composition
  \[
  \xymatrix{ D^b(\tilde{X}) \ar[r]^(.6){(i_-^\star)^{-1}} & \widetilde{\mathbf{H}} \ar[r]^(.4){i_+^\star}  &D^b(X) }
  \]
  is equal to the push-down functor $f_\star$.
\end{enumerate}
\end{lemma}
\begin{proof}
  (1) We use the diagram \eqref{eq:vGIT_over_X}. If $\cF$ is any sheaf on $X$, then $\pi^\star \cF\big|_Z$ is of $\Cstar$-weight zero, and so $\pi^\star\cF\in \mathbf{H}$. Moreover, since $i_+^\star \pi^\star$ is the identity functor, we must have that $\pi^\star$ is an embedding and
  \[
  \pi^\star ( D^b(X) )= \mathbf{H}
  \]
  with $\pi^\star= (i_+^\star)^{-1}$. Now the statement follows, since $i_-^\star \pi^\star=f^\star$.

  (2) Let $\cE\in \mathbf{H}$ and $\cF\in \widetilde{\mathbf{H}}$. By \cite[Theorem 3.29]{Halpern-Leistner}, restriction gives a quasi-isomorphism
  \[
  R\Hom_{[M/\Cstar]} (\cE, \cF) \stackrel{\sim}{\longrightarrow} R\Hom_{X}(i_+^\star \cE, i_+^\star \cF) 
  \]
  In other words, the composition:
  \[
  \xymatrix{ \widetilde{\mathbf{H}} \ar[r]^(.4){i_+^\star} &  D^b(X) \ar[r]^(.6){(i_+^\star)^{-1}} &   \mathbf{H}}
  \] 
  is the right adjoint to the inclusion $\mathbf{H} \hookrightarrow \widetilde{\mathbf{H}}$. If we identify $\mathbf{H}$ and $\widetilde{\mathbf{H}}$ with $D^b(X)$ and $D^b(\tilde{X})$ using $i_+^\star$ and $i_-^\star$ respectively, then the inclusion $\mathbf{H} \hookrightarrow \widetilde{\mathbf{H}}$ is identified with $f^\star$ by (1), and so its right adjoint must coincide with $f_\star$.  
\end{proof}

\subsection{The Fourier-Mukai Functor and Variation of GIT} 

In this section we complete the proof that the Fourier--Mukai functor $\mathbb{FM}$ arising from the diagram \eqref{eq:K-equivalence_again} is a derived equivalence, by showing that it coincides with the `grade-restriction' derived equivalance $\mathbb{GR}$.

\subsubsection{Variation of GIT Setup}

We saw in \S\ref{sec:GIT_variation} that $X_+$,~$X_-$ and~$\tilde{X}$ can be constructed using a single GIT problem. These quotients correspond respectively to chambers which we denoted $\widetilde{C}_+$,~$\widetilde{C}_-$~and $\widetilde{C}$. Let $W_{+|-}$, $W_{+|\sim}$ and $W_{-|\sim}$ denote the codimension-1 walls between these three chambers, and let $W_0$ be the codimension-2 wall where all three meet. The three codimension-1 walls each define one-parameter subgroups of $K\times \Cstar$, which have fixed loci, repelling subvarieties, and attracting subvarieties as follows.
\begin{align*}
 \text{Wall:}		 &&&     W_{+|-}				 && W_{-|\sim}				 && W_{+|\sim}   \\
 \text{One-parameter subgroup:} &&&		(e,0)				&& (0,1) 					&& (e,1)	\\
 \text{Fixed locus:} &&&	 \CC^{M_0}\times \CC		&&  \CC^{M_{\leq 0}}	&&  \CC^{M_{\geq 0}}\\	
 \text{Repelling subvariety:} &&&	\CC^{M_{\geq 0}} \times \CC 		&& \CC^{M_{\leq 0}} \times \CC		&&  \CC^{M_{\geq 0}}\times \CC  \\
 \text{Attracting subvariety:}&&&	\CC^{M_{\leq 0}} \times \CC		&& \CC^m 					&& \CC^m
\end{align*}
Consider~7 stability conditions as follows: one lying on (the relative interior of) $W_0$, one lying on (the relative interior of) each of the 3 codimension-1 walls, and one lying in each chamber. The semi-stable locus $V_0 \subset \CC^{m+1}$ for a stability condition lying on $W_0$ is the open set in $U_0\times \CC$:
\[
V_0 = \left( 
\bigcup_{I\in \cA_0, I \subset M_0} (\Cstar)^I \times \CC^{\overline{I}}\right) 
\times \CC 
\cup 
\left( \bigcup_{I\in \cA_0, I \not\subset M_0} 
(\Cstar)^I \times \CC^{\overline{I}}\right) \times \Cstar 
\]
where $\overline{I}$ denotes the complement of $I$ in $\{1,2,\dots,m\}$ 
and $U_0$ was defined in Section~\ref{sect:GRR}. 
The semi-stable locus for the other~6 stability conditions are open subsets of $V_0$, as follows:
\begin{center}
  \begin{tabular}{ll}
    \multicolumn{1}{c}{Location of  stability condition} & \multicolumn{1}{c}{Semi-stable locus}   \\
    \hspace{2.25cm}$\widetilde{C}_+$ 	& \hspace{2ex}	$V_+ = V_0\setminus \left((\CC^{M_{\leq 0}}\times \CC)\cup \CC^m\right)	$		\\
    \hspace{2.25cm}$\widetilde{C}_-$ 	& \hspace{2ex}	$V_- = V_0\setminus \left((\CC^{M_{\geq 0}}\times \CC)\cup \CC^m\right)	$				\\
    \hspace{2.25cm}$\widetilde{C}$ 	& \hspace{2ex}	$V_{\sim} = V_0\setminus \left((\CC^{M_{\leq 0}}\times \CC)\cup (\CC^{M_{\geq 0}}\times \CC)\right)$	\\
    \hspace{2.25cm}$W_{+|-}$		& \hspace{2ex}	$V_{+|-} = V_0\setminus \CC^m$		\\
    \hspace{2.25cm}$W_{+|\sim}$ 	& \hspace{2ex}	$V_{+|\sim}= V_0\setminus (\CC^{M_{\leq 0}}\times \CC)$		\\
    \hspace{2.25cm}$W_{-|\sim}$ 	& \hspace{2ex}	$V_{-|\sim} = V_0\setminus (\CC^{M_{\geq 0}}\times \CC)$			
  \end{tabular}
\end{center}
The GIT quotients $\big[V_+ /K\big]$, $\big[V_- /K\big]$, and $\big[V_{\sim} /K\big]$ are $X_+$, $X_-$, and $\tilde{X}$ respectively.

Let $k_i = \max(D_i \cdot e,0)$.  The maps:
\begin{align*}
    \bar{\pi}_- \colon \CC^{m+1} & \to \CC^m & T \times \Cstar & \to T \\
    (x_1,\ldots,x_{m+1}) & \mapsto \big(x_1 x_{m+1}^{k_1},\ldots,x_m x_{m+1}^{k_m}\big) &
  (\theta,\theta') & \mapsto \theta
\end{align*}
induce a morphism $\pi_- \colon \big[\CC^{m+1} / (T \times \Cstar) \big] \to \big[ \CC^m / T \big]$. This morphism maps the subset $V_0$ to the subset $U_0$, and it maps the subset $V_{-|\sim}$ to the subset $U_-$. Thus we have a commutative diagram:
\[ \xymatrix{
  & &    \big[V_{-|\sim}  / (T\times \Cstar) \big]  \ar[dd]^{\pi_-} & \\
  \big[X_-/Q\big] \ar@{=}[r] & \big[V_- / (T \times \Cstar) \big] \ar@{^{(}->}[ur] \ar@{=}[dr] & & 
  \big[V_{\sim} / (T \times \Cstar) \big] \ar@{_{(}->}[ul]  \ar@{=}[r] \ar[ld] & \big[\widetilde{X}/Q\big] \ar@/^/[lld]^{f_-}  \\
  & & \big[X_-/Q\big]}
\]
where $f_-$ is the ($Q$-equivariant) blow-up. Similiarly there is a map $\pi_+$ which sends $V_{+|\sim}$ to $U_+$ and gives a corresponding commutative diagram for $f_+$.  The stack $\big[ V_{+|-} / (T\times \Cstar)\big]$ is isomorphic to $\big[U_0 / T \big]$, via either of $\pi_-$ or $\pi_+$.

\subsubsection{Proof that $\FM$ Coincides With $\mathbb{GR}$}

As discussed, the fact that $\FM$ is an equivalence follows from:
\begin{proposition} The two functors
  \begin{align*}
    \FM \colon D^b_Q(X_-) \longrightarrow D^b_Q(X_+) && \text{and} && 
    \mathbb{GR} \colon D^b_Q(X_-) \longrightarrow D^b_Q(X_+)
  \end{align*} 
  are naturally isomorphic. 
\end{proposition}

\begin{proof}
Consider the following poset of inclusions:
\begin{equation}\label{eq:poset}
\begin{aligned}
 \xymatrix{ & & V_0  & & \\
& V_{-|\sim} \ar[ur]  & & V_{+|\sim} \ar[ul] & \\
V_- \ar[ur] & & V_\sim\ar[ul]\ar[ur]  & & V_+\ar[ul]   }
\end{aligned}
\end{equation}
Passing to (equivariant) derived categories, we get a corresponding commuting diagram of restriction functors. We will prove our proposition by lifting the categories along the bottom line up the diagram, using grade-restriction rules.

 Let us denote by $d$ the positive integer:
\[
d = \sum_{i\in M_+} D_i \cdot e = {- \sum_{i\in M_-}} D_i \cdot e 
\]
We begin by considering $V_+$ and $V_\sim$ as open subsets of $V_{+|\sim}$. They are the complements, respectively, of the KN stratum:
\[
\Big( (e,1), \;\; \CC^{M_{\geq 0}} \cap V_{+|\sim}, \;\; \CC^m \cap V_{+|\sim} \Big)
\]
which has numerical invariant $\eta = 1$, and the KN stratum:
\[
\Big( (-e,-1), \;\;  \CC^{M_{\geq 0}}\cap V_{+|\sim}, \;\; (\CC^{M_{\geq 0}}\times \CC) \cap V_{+|\sim} \Big)
\]
which has numerical invariant $\eta = d$.  Hence we define subcategories 
\[
\mathbf{F}\subset \widetilde{\mathbf{F}} \subset \Dstack{V_{+|\sim}}{T\times \Cstar} 
\]
by imposing the grade-restriction rule \eqref{eq:GRR} on the subvariety $ \CC^{M_{\geq 0}}\cap V_{+|\sim}$, where for $\mathbf{F}$ we require that the $(e,1)$-weights lie in the interval $[0,1)$, and for $\widetilde{\mathbf{F}}$ we require that the $(e,1)$-weights lie in the interval $[0, d)$. Then $\mathbf{F}$ is equivalent under restriction to $D^b_Q(X_+)$, and $\widetilde{\mathbf{F}}$ is equivalent under restriction to $D^b_Q(\tilde{X})$. Using the map $\pi_+$, and arguing exactly as in Lemma \ref{lem:blowupsandGRR}, we have a commuting triangle
$$\xymatrix{ & \widetilde{\mathbf{F}} \ar[dl]_{\simeq} \ar[dr] & \\
D^b_Q(\tilde{X}) \ar[rr]^{(f_+)_\star} & & D^b_Q(X_+) } $$
where the diagonal maps are the restriction functors.

Now view $V_-$ as an open subset of $V_{-|\sim}$, where it is the complement of the KN-stratum:
\[
\Big( (0,1), \;\; \CC^{M_{\leq 0}}\cap V_{-|\sim}, \;\; \CC^m\cap V_{-|\sim} \Big)
\]
which has numerical invariant $\eta = 1$.  Hence we define a subcategory
\[
\mathbf{H}\subset \Dstack{V_{-|\sim}}{T\times \Cstar}
\]
using the grade restriction rule on the subvariety $\CC^{M_{\leq 0}}\cap V_{-|\sim}$ and requiring the $(0,1)$-weights to lie in the interval $[0,1)$. Then $\mathbf{H}$ is equivalent under restriction to $D^b_Q(X_-)$ (there is also a larger subcategory $\widetilde{\mathbf{H}}$ which is equivalent to $D^b_Q(\tilde{X})$, but we will not need this). Using the map $\pi_-$ and the proof of Lemma \ref{lem:blowupsandGRR} again, we have commuting triangle
$$\xymatrix{ & \mathbf{H} \ar[dl]_{\simeq} \ar[dr] & \\
D^b_Q(X_-) \ar[rr]^{(f_-)^\star} & & D^b_Q(\tilde{X}) } $$
where the diagonal maps are the restriction functors.

Next we recall the definition of the functor $\mathbb{GR}$ from Section \ref{sect:GRR}. It is constructed by lifting $D^b_Q(X_-)$ to a subcategory $\mathbf{G}\subset \Dstack{U_0}{T}$ and then restricting to $\big[X_+/Q\big]$. Consider the subcategory:
\[
(\pi_-)^\star \mathbf{G} \subset \Dstack{V_0}{T\times \Cstar}
\]
Since we have a commuting diagram
\[
\xymatrix{
  & \big[V_- / (T\times \Cstar) \big]  \ar@{^{(}->}[r] \ar@{=}[d]  &
  \big[V_{+|-} / (T\times \Cstar) \big]  \ar@{^{(}->}[r] \ar@{=}[dr]  &  \big[V_0 / (T\times \Cstar) \big] \ar[d]^{\pi_-} \\
  \big[X_-/Q\big] \ar@{=}[r] & \big[U_- / T \big] \ar@{^{(}->}[rr] &&  \big[U_0 / T \big] 
}
\]
the subcategory $(\pi_-)^\star \mathbf{G}$ must be equivalent to $D^b_Q(X_-)$ under restriction, and therefore we can also obtain the functor $\mathbb{GR}$ by inverting this equivalence and then restricting to $D^b_Q(X_+)$. 

Now take an object $\cE\in (\pi_-)^\star \mathbf{G}$. From the definition of $\mathbf{G}$, and the fact that 
\[
\CC^{M_{\geq 0}} \;\subset \;(\pi_-)^{-1}\left( \CC^{M_0}\right) 
\]
it follows that the homology sheaves of the restriction of $\cE$ to $\CC^{M_{\geq 0}}\cap V_0$ have $(e,1)$-weights lying in the interval $[0, d)$. Consequently, the restriction functor from $V_0$ to the open subset $V_{+|\sim}$ maps the subcategory $(\pi_-)^\star \mathbf{G}$ into the subcategory $\widetilde{\mathbf{F}}$. 

Also, the homology sheaves of the restriction of $\cE$ to $\CC^{M_{\leq 0}} \cap V_0$ have $(0,1)$-weight zero, since this is true of any object in the image of $(\pi_-)^\star$. Consequently the restriction functor from $V_0$ to $V_{-|\sim}$ maps $(\pi_-)^\star \mathbf{G}$ into $\mathbf{H}$. This must in fact be an equivalence, since both categories are equivalent to $\Dstack{X_-}{Q}$ under restriction to $V_-$.

Putting all of the above together, we have a commutative diagram
\[ \xymatrix{ & & (\pi_-)^\star\mathbf{G} \ar[dl]_{\simeq} \ar[dr] & & \\
& \mathbf{H} \ar[dl]_{\simeq} \ar[dr] & & \widetilde{\mathbf{F}} \ar[dl]_{\simeq} \ar[dr] & \\
D^b_Q(X_-) \ar[rr]^{(f_-)^*} & & D^b_Q(\tilde{X})\ar[rr]^{(f_+)_\star} & & D^b_Q(X_+)}
\]
in which all the downward arrows are restriction functors (cf.~the diagram \eqref{eq:poset}).  We conclude that the functor $\mathbb{GR}$ agrees with the composition $(f_+)_\star (f_-)^\star$, which is the statement of the Proposition.
\end{proof}

\section*{Acknowledgements}

H.I.~thanks Hiraku Nakajima for discussions on the equivariant index theorem.  This research was supported by a Royal Society University Research Fellowship; the Leverhulme Trust; ERC Starting Investigator Grant number~240123; EPSRC Mathematics Platform grant EP/I019111/1; JSPS Kakenhi Grant Number 16K05127, 16H06337, 25400069; NFGRF, University of Kansas; Simons Foundation Collaboration Grant 311837; and an Imperial College Junior Research Fellowship.

\bibliographystyle{plain}
\bibliography{bibliography}

\end{document}